\documentclass[11pt,twoside]{article}
\usepackage[mathscr]{eucal}
\usepackage{amsbsy}
\usepackage{amssymb}
\usepackage{amsmath}
\usepackage{amsthm}
\usepackage{graphics}

\newcommand{\brc}[1]{\left\lbrace #1 \right\rbrace}
\newcommand{\abs}[1]{\vert #1 \vert}

\newcommand{\absss}[1]{\Bigl\vert #1 \Bigr\vert}

\newcommand{\norm}[1]{\Vert #1 \Vert}

\def\phi{\varphi}
\def\bbn{{\mathbb N}}
\def\bbr{{\mathbb R}}
\def\bbh{{\mathbb H}}

\def\a{\alpha}
\def\b{\beta}
\def\d{\delta}
\def\e{\varepsilon}
\def\g{\gamma}
\def\G{\Gamma}
\def\l{\lambda}
\def\k{\kappa}
\def\z{\zeta}
\def\s{\sigma}

\def\om{\omega}
\def\vp{\varphi}

\def\ll{{\bold l}}

\def\wh{\widehat}

\def\H{{\mathscr H}}
\def\=A8{\"o}

\theoremstyle{plain}
\newtheorem{teo}{Theorem}

\newtheorem{prop}[teo]{Proposition}
\theoremstyle{definition}

\newtheorem{example}{Example}
\newtheorem{remark}[teo]{Remark}

\title{On small balls problem for stable measures in a Hilbert space
\footnotemark[0]\footnotetext[0]{%
\textit{MSC 2000 subject classifications}. Primary-60B11, 60E07 .}
\footnotemark[0]
\author{\sc  Vygantas {\sc Paulauskas}}}

\date{\it Vilnius University,
Department of Mathematics and Informatics }

\begin{document}

\maketitle
\begin{abstract}
In the paper the old results on probabilities of small balls for
stable measures in a Hilbert space, obtained in 1977 and remaining
unpublished, are presented. Apart of historical value these results
are interesting even now, since they are comparable with recently
obtained ones.
\end{abstract}
\vfill
\eject


\section{Introduction}
The story of this paper is rather unusual, it contains the results
obtained  during my stay at Gothenburg university during 1976-77
academic year. During this year I was dealing with infinitely
divisible and stable measures on Hilbert and Banach spaces
(description and properties of these measures, rates of the
convergence to stable laws, small ball problem and Law of the
iterated law for stable measures ) and all obtained results were
presented as two Chalmers university of Technology and University of
G{\" o}teborg preprints (see \cite{paulIII} and \cite{paulIV}). At
this time as a Soviet Union citizen I had almost no experience of
publishing papers outside Soviet Union (the only one paper
\cite{paul} appeared abroad, since prof P.R. Krishnaiah, as Editor
of J. of Multivariate Analysis, being in Vilnius took my manuscript
with himself to USA), so I brought these preprints to Vilnius. But
there were new results to be published, therefore results from
preprints remained unpublished, only four years later part of the
results were incorporated in papers \cite{paulV},\cite{paulrac}. The
rest of results, among them results on small balls for stable
measures, remained unpublished till now. In May 2008 at the
conference "High dimensional probability" in Luminy after Mikhail
Lifshits survey talk on small ball problem I had mentioned to him
that in 1977 I had dealt a little bit with this problem for stable
measures in a Hilbert space. After returning to Vilnius I sent to
St. Petersburg preprint \cite{paulIV}. After some time I got
Mikhail's e-mail saying that these results are not only interesting
from historical perspective, but some of them are comparable with
the results obtained recently. He put this preprint on the
bibliography list on small balls problem (see \cite{lif}) and
encouraged me to prepare a paper and to put it into ARXIV. But again
preparing the paper took a year... Also here it is necessary to
note, that J. Hoffmann-Jorgensen's preprint \cite{hofmjorg}was
inspiring my investigations in 1977, and at this time for me it was
the only one source on the problem (later most of the results of
this preprint were included in \cite{hofmandudley}). As a matter of
fact, looking at the bibliography \cite{paulIV} one can find only
few papers which appeared before 1975. Also the paper is unusual
since there is no comparison of our estimates  with recently
obtained results. I had not followed the development of this field
during last decades and such task would be too difficult for me.

Thus the next section contains almost unchanged text from
\cite{paulIV}, only small changes or explanations were made in order
to make the text understandable.

\section{Results}
 Through all the paper  ${\bbh}$ will stand for  a real
separable Hilbert space with a norm $||\cdot ||$ and a scalar
product $(\cdot,\cdot).$ Without loss of generality we shall take
$\bbh =\ll_2,$ thus for $x=(x_1, \dots, x_n,\dots,) \ y=(y_1,
\dots, y_n,\dots,)(x,y)=\sum_{i=1}^\infty x_iy_i.$ Also we shall
use the norm $||x||_\infty=\sup_i|x_i|.$ We denote $S_r(a)=\{x\in
\bbh: ||x-a||\le r \}, \ a\in \bbh, \ S_r=S_r(0), U_r=\{x\in \bbh:
||x||= r \}, \quad D_a=\{x\in \bbh: ||x||_\infty \le a, \ a>0 \}.$

 The method, which we use is the
same as in \cite{hofmjorg}, where bounds for small balls were
obtained in the Gaussian case.

Let us have a stable distribution $ \mu $ with a characteristic
function (ch.f.)

\begin{align}
\label{mod1} \wh{\mu}(y) = \exp \brc{-\frac{1}{2}
\sum_{i=1}^{\infty} \abs{y_{i}}^{\a} \l_{i}},
\end{align}
where $ \l = (\l_{1}, \l_{2}, ...) \in l_{1}  $, $  \l_{1} \geq
\l_{2} \geq , ..., \geq 0 $ and let $ \s_{\a}(x) $ be a
differentiable decreasing function on $ [1,\infty) $ such that $
\s_{\a}(n) \leq \l_{n}^{\frac{1}{\a}}, n = 1,2,... $

Let us denote by $p_\a(x)$ a density function of an
one-dimensional symmetric stable law with the ch.f.
$g_\a(t)=\exp\{-|t|^\a/2\}$ and let
$p_{j,\a}(x)=\l_j^{-1/\a}p_\a(x\l_j^{-1/\a})$. Let $ \vp(x)=(2
p_{\a}(0) \s_{\a}(x))^{-1} $ and $ \nu: \bbr_{+} \rightarrow \bbn
$ (= the set of positive integers) such, that

\begin{align*} \e\vp(\nu(\e)) \leq 1 \end{align*} for all $ 0 < \e < (\vp(1))^{-1}
$. Let

\begin{align*}
H(y) = \int_{1}^{y} \dfrac{x\vp'(x)}{\vp(x)} dx
\end{align*}

We put $ \Psi(x) = C_{0}(\a) (\s_{\a}(x) \sqrt{x})^{-1} $, where $
C_{0}(\a) $ is given in \eqref{mod8}, $ \eta: \bbr_{+} \rightarrow
\bbn $ such that $ \e\Psi(\eta(\e)) \leq 1 $, and

\begin{align*}
K(y) = \int_{1}^{y} \dfrac{x\Psi'(x)}{\Psi(x)} dx
\end{align*}

\begin{prop}
\label{prop1} For all $ 0 < \a \le 2 $ and $ 0 < \e < (\max
\vp(1), \Psi(1))^{-1} $ the following estimates hold:
  \begin{align}
  \label{mod2}
     \mu(D_{\e}) \leq C(\a ,\l_{1}) \exp \brc{-H(\nu(\e)) -
     \log(\e)},
  \end{align}

  \begin{align}
  \label{mod3}
     \mu(S_{\e}) \leq C(\a ,\l_{1}) \exp \brc{-K(\eta(\e)) -
     \log(\e)},
  \end{align}
and, if $ \log \s_{\a}(x) $ is convex, then
  \begin{align}
  \label{mod4}
     \mu(D_{\e}) \leq C(\a ,\l_{1}) \exp \brc{-H(\nu(\e)) - \frac{1}{2} \log
     \e},
  \end{align}

  \begin{align}
  \label{mod5}
     \mu(S_{\e}) \leq C(\a ,\l_{1}) \exp \brc{-K(\eta(\e)) - \frac{1}{4}
     \log(\e^{2}\eta(\e))}.
  \end{align}
\end{prop}

\begin{proof}
Let $\pi_n: l_2 \to R_n$ be the usual projection, defined by
$\pi_n(x)=(x_1, \dots , x_n)$. Let
$\mu^{(n)}(A)=\mu(\pi_n^{-1}A),$ where $A$ is a Borel set in $R_n$
and $\pi_n^{-1}A$ stands for a cylindrical set in $l_2$ with a
base $A$, $D_a^{(n)}=\pi_nD_a=\{x\in R_n: |x_i|< a, \ a>0, \
i=1,\dots , n \}$

We have for all $ n \geq 1 $

\begin{align}
\label{mod6}
\mu(D_{\e}) &\leq \mu^{(n)}(D_{\e}^{(n)}) = \prod_{j=1}^{n} (2\int_{0}^{\e} p_{j,\a}(x)dx)
\leq \prod_{i=1}^{n} (2 p_{\a}(0) \e \l_{i}^{-\frac{1}{\a}})
\nonumber \\
&\leq \exp\brc{n \log \e + n \log(2p_{\a}(0)) - \sum_{i=1}^{n} \log(\s_{\a}(i))}
\nonumber \\
&\equiv \exp\brc{F(n,\e,\s)}
\end{align}

Since $ \s_{\a}(x) $ is decreasing

\begin{align*}
F(n, \e, \s) & \leq - \int_{1}^{n} \log \s_{\a}(x) dx - \log \s_{\a}(n) + \int_{1}^{n} \log(2 p_{a}(0)) dx \\
&+ \log(2 p_{\a}(0) + n \log \e = \int_{1}^{n} \log \vp(x) dx + \log \vp(n) + n \log \e \\
&= n \log \e \vp(n) - \log \vp(1) - H(n) + \log \vp(n)
\end{align*}

From this estimate and  \eqref{mod6} we get \eqref{mod2}. If $
\log \s_{\a}(x) $ is convex function, then instead of inequality

\begin{align*}
 - \sum_{i=1}^{n} \log \s_{\a}(i) \leq - \int_{1}^{n} \log \s_{\a}(x) dx - \log
 \s_{\a}(n),
\end{align*}
we can use the inequality

\begin{align*}
- \sum_{i=1}^{n} \log \s_{\a}(i) & = - \sum_{i=1}^{n-1} \frac{1}{2} (\log \s_{\a}(i) + \log \s_{\a}(i+1)) \\
&- \frac{1}{2} \log \s_{\a}(1) - \frac{1}{2} \log \s_{\a}(n) \\
&\leq - \int_{1}^{n} \log \s_{\a}(x) dx - \frac{1}{2} \log \s_{\a}(1) - \frac{1}{2} \log \s_{\a}(n)
\end{align*}
and, arguing as before, we get \eqref{mod4}.

Now we consider $ S_{\e} $ and if we put $ S_{\e}^{n} = \brc{x \in \bbr_{n}: \sum_{i=1}^{n} x_{i}^{2} < \e^{2}} $, $ \overline{\l}_{n}(dx) $ - the Lebesque measure in $ \bbr_{n} $, $ V_{n} = \overline{\l}_{n} (S_{1}^{(n)}) = \pi^{n/2} \left( \G \left(\dfrac{n+2}{2} \right) \right)^{-1} $, we have for all $ n \geq 1 $

\begin{align*}
 \mu(S_{\e}) & \leq \mu^{(n)} (S_{\e}^{(n)}) = \int_{\sum_{i=1}^{n} x_{i}^{2} \leq \e^{2}}
  \prod_{i=1}^{n} p_{i,\a}(x_{i}) \overline{\l}_{n}(dx) \\
  & \leq \e^{n} V_{n} \left(\prod_{i=1}^{n} \l_{i}^{1/\a}\right)^{-1}
  (p_{\a}(0))^{n}.
\end{align*}
Using Stirling's formulae and the equality

\begin{align*}
p_{\a}(0) = \dfrac{1}{2\pi} \int_{-\infty}^{\infty} \exp
(-\frac{1}{2} \abs{t}^{\a}) dt = 2^{1/\a}
\G\left(\frac{1}{\a}\right) (\pi\a)^{-1},
\end{align*}
we get
\begin{align}
\label{mod7}
& \mu(S_{\e}) \leq C \e^{n} \left(\prod_{i=1}^{n} \l_{i}^{1/\a}\right)^{-1} (C_{0}(\a))^{n} \frac{1}{\sqrt{n}} \left( \frac{n}{e} \right)^{n/2}
\nonumber \\
& = C \exp \brc{-\sum_{i=1}^{n} \log \s_{\a}(i) + n \log
(C_{0}(\a)\e)  - \dfrac{1}{2} \log n - \dfrac{n}{2}(\log n - 1)},
\end{align}
where
\begin{align}
\label{mod8} C_{0}(\a) = \dfrac{2^{1/2+1/\a}
\G\left({\a}^{-1}\right)}{\sqrt{\pi}\a}.
\end{align}
Now we have to estimate the expression

\begin{align*}
F_{1}(n, \e, \s) = n \log \e + n \log C_{0}(\a) - \sum_{i=1}^{n}
\log \s_{\a}(i) - \frac{1}{2} \log n - \frac{n}{2} (\log n - 1).
\end{align*}
By means of the equality

\begin{align*}
\frac{n}{2} (\log n - 1) = \int_{1}^{n} \log \sqrt{x} dx -
\frac{1}{2},
\end{align*}
after some calculations we get

\begin{align}
\label{mod9}
F_{1}(n, \e, \s) \leq n \log(\e \eta(\e)) - \log \eta(1) - K(n) + \log \eta(n) + \frac{1}{2}
\end{align}

From \eqref{mod7} and \eqref{mod9}  inequality \eqref{mod3}
follows. The relation \eqref{mod5} can be obtained in a similar
way by the use of the convexity of $ \log \s_{a}(x) $.
\end{proof}

\begin{remark}
Since $ C_{0}(2) = 1 $ it is easy to see that in the case $ \a = 2
$ the inequalities \eqref{mod3} and \eqref{mod5} coincide with
formulas  (2.1.4) and (2.1.6) with $ a = 0 $ in \cite{hofmjorg}.
\end{remark}

Now we shall deal with another class of stable distributions on H. Let $ \nu $ be a symmetric stable distribution with ch. f.
\begin{align}
\label{mod10} \wh{\nu}(y) = \exp \brc{-\frac{1}{2}(Ty, y)^{\a/2}},
0 < \a \leq 2
\end{align}
where $ T $ is some positive trace-class operator. Since now we
shall consider probabilities of balls, it is obvious that, without
loss of generality, we may assume that $ T $ is diagonal with the
numbers $ \l_{1} \geq \l_{2} \geq ... > 0 $ on the diagonal and
\begin{align}
\label{mod11} \wh{\nu}(y) = \exp
\brc{-\frac{1}{2}\left(\sum_{i=1}^{n} \l_{i}
y_{i}^{2}\right)^{\a/2}}.
\end{align}

May be it is worth to mention, that we cannot point out the spectral measure $ \G $ on $ U_{1} $ to which the ch.f \eqref{mod10}
or \eqref{mod11} corresponds, but it is easy to verify directly, that \eqref{mod11} is a positive define function,
continuous in the S-topology with $ \wh{\nu}(0) = 1 $ and $ \wh{\nu}(ay) \cdot \wh{\nu}(by)
= \wh{\nu} \left( \left(a^{\a} + b^{\a} \right)^{1/\a} y \right) $. Thus it is the ch. f. of some stable distribution on $ H $.
 We can only say, that this measure $ \G $ is not discrete. Moreover, it seems likely that in this case it is more convenient
  to have a spectral measure not on the unit sphere $ U_{1} $, but on the ellipsoid $ \brc{y \in H: \sum_{i=1}^{\infty} y_{i}^{2} \l_{i}^{-1} = 1} $
 for the following reason. If we have a stable distribution in
   $ \bbr_{n} $ with the ch.f. $ \vp(t) = \exp \brc{-\frac{1}{2} \left( \sum_{i=1}^{n} t_{i}^{2} \right)^{\a/2}} $,
   then from \cite{paul} we know, that
\begin{align*}
\vp(t) = \exp\brc{-\int_{U_1^{(n)}} \abs{(t,x)}^\a \om_{n}(dx)}
\end{align*}
where $U_1^{(n)}=\{ x\in R_n: \sum_{i=1}^{n} x_{i}^{2}=1\}$  and $
\om_{n} $ is the Lebesque measure on the unit sphere $U_1^{(n)}$.
Then
\begin{align*}
\exp \brc{-\frac{1}{2}\left(\sum_{i=1}^{n} \l_{i} t_{i}^{2}\right)^{\a/2}} & = \exp \brc{- \int_{U_1^{(n)}}
\abs{\sum_{i=1}^{n} t_{i} \sqrt{\l_{i}} x_{i}}^{\a} \om_{n}(dx)} \\
& = \exp \brc{- \int_{U_{1,\l}^{(n)}} \abs{(t,y)}^{\a}
\overline{\om}_{n}(dy)}
\end{align*}
where $ \overline{\om}_{n} $ is the measure on the ellipsoid
$U_{1,\l}^{(n)}:=\{y \in R_n: \sum_{i=1}^{n} y_{i}^{2} \l_{i}^{-1} =
1\} $, obtained from the Lebesque measure on the unit sphere by the
variable substitution $ x_{i} \sqrt{\l_{i}} = y_{i} $, $ i =
1,2,...,n $. Therefore it is possible to say that to the ch. f.
\eqref{mod11} there corresponds a measure $ \G $ on the ellipsoid $
U_{1,\l}=\{\sum_{i=1}^{n} y_{i}^{2} \l_{i}^{-1} = 1\} $ such that
it's projections to n-dimensional ellipsoids $ U_{1,\l}^{(n)} $.
Here it is necessary to explain how projections of $ \G $ must be
understood. If we have a stable measure $\mu$ on $l_2$ with an
exponent $\alpha$ and a spectral measure $\Gamma$ on $U_1$, then we
have a L{\' e}vy measure $M$ defined by formula $M(dx)=r^{-(1+\a)}dr
\Gamma(ds),$ where $r=||x||, s=x\||x||$. Denote by $\Pi_k$ usual
projection defined by $\Pi_k x=(x_1, \dots , x_k)$,  and take $\Pi_k
M$. Since this L{\' e}vy measure corresponds to a stable
distribution on $ \bbr_{n}$, therefore $\Pi_kM(dy)={\tilde
r}^{-(1+\a)}d{\tilde r} {\tilde \Gamma}(d{\tilde s}),$ where
$y=(y_1, \dots ,y_k), \ {\tilde r}=||y||, \ {\tilde s}=y/||y||.$ A
measure ${\tilde \Gamma}$ is a projection of $\Gamma$ and is denoted
by $\Pi_k \Gamma.$

In what follows by $ \nu \equiv \nu(\a, \l) $, $ \l \in l_{1}^{+}
$, $ 0 < \a \leq 2 $ we shall denote the stable measure with ch.
f. \eqref{mod11}. Let $ \overline{\s}(x) $, $ x \in [1,\infty) $
be a differentiable decreasing function such that $
\overline{\s}(j) \leq \l_{j}^{1/2} $, $ j \geq 1 $. Let $ \k(x) =
\left( 2/{\a}\right)^{1/\a} x^{(1-\a)/{\a}}
\left(\overline{\s}(x)\right)^{-1} $,  $ \z: \bbr_{+} \rightarrow
\bbn $ such that $ \e \k\left(\z(\e)\right) \leq 1 $ for $ 0 < \e
< \left(\k(1)\right)^{-1} $ and
\begin{align*}
 L(y) = \int_{1}^{y} \frac{x \k'(x)}{\k(x)} dx.
\end{align*}

\begin{prop}
\label{prop2}
For all $ 0 < \a \leq 2 $ and $ 0 < \e < \left(\k(1)\right)^{-1} $
 \begin{align}
 \label{mod12}
   \nu(S_{\e}) \leq C(\a, \l_{1}) \exp \brc{-L\left( \z(\e)\right) - \frac{2-\a}{2\a} \log \z(\e) - \log \e}
 \end{align}
 and, if $ \log \overline{\s}(x) $ is convex, then
\begin{align}
 \label{mod13}
 \nu(S_{\e}) \leq C(\a, \l_{1}) \exp\brc{-L\left(\z(\e)\right) - \frac{1}{2\a} \log\left( \e^{\a}
 \z(\e)\right)}.
 \end{align}
\end{prop}

\begin{proof}
Let $ q_{n,\a}(x) $, $ x \in \bbr_{n} $ be the density of the
n-dimensional symmetric stable law with the ch. f. $ \exp
\brc{-\frac{1}{2} \left(\sum_{i=1}^{n} t_{i}^{2} \right)^{\a/2}} $
and let  $ q_{n,\a,\l}(x) $ be a density,  corresponding  to the
ch. f. $ \exp \brc{-\frac{1}{2} \left(\sum_{i=1}^{n} t_{i}^{2}
\l_{i} \right)^{\a/2}} $. Then

$$
q_{n,\a,\l}(x) = \dfrac{1}{\left( \prod_{1}^{n}
\l_{i}\right)^{1/2}} q_{n,\a} \left( \frac{x_{1}}{\sqrt{\l_{1}}},
..., \frac{x_{n}}{\sqrt{\l_{n}}} \right),
$$
\begin{align}
\label{mod14}
\abs{q_{n,\a}(x)} & \leq q_{n,\a}(0) = \frac{1}{(2\pi)^{n}} \int_{\bbr_{n}} \exp \brc{-\frac{1}{2} \norm{t}^{\a}} dt \nonumber \\
& = \frac{2}{(2\pi)^{n}} \frac{\pi^{n/2}}
{\G\left(\frac{n}{2}\right)} \int_{0}^{\infty}
e^{-\frac{r^{\a}}{2}} r^{n-1} dr = \dfrac{2^{1+\frac{n}{\a}}
\pi^{n/2} \G\left( \frac{n}{\a}\right)} {(2\pi)^{n}
\G\left(\frac{n}{2}\right) \a}.
\end{align}
In the last equality we have used the well-known formula
\begin{align*}
 \int_{0}^{\infty} e^{-x^{r}} x^{p} dx = \frac{1}{r} \G\left(\frac{p+1}{r} \right) \text{ , } p > -1 \text{, } r >
 0.
\end{align*}
Now applying \eqref{mod14} we get
\begin{align*}
 \nu\left(S_{\e}\right) & \leq \int_{S_{\e}^{(n)}} q_{n,\a,\l}(x) dx \leq V_{n} \e^{n} \sup q_{n,\l, \a} (x) \\
 & \leq V_{n} \e^{n} \dfrac{2^{\frac{n+\a}{\a}} \pi^{\frac{n}{2}} \G\left( \frac{n}{\a}\right)}{\left( \prod_{i}^{n} \l_{i}
  \right)^{1/2} \a (2 \pi)^{n} \G \left( \frac{n}{2} \right)}.
\end{align*}

Recalling the value of $ V_{n} $ and applying Stirling's formulae,
after some steps which are omitted, we arrive at the following
estimate
\begin{align}
\label{mod15} \nu \left(S_{\e} \right) \leq C(\a) \exp
\brc{F_{2}(n, \e, \overline{\s}},
\end{align}
where
\begin{align*}
 F_{2}(n,\e,\overline{\s}) & = n \log \e + \frac{n}{\a} \log \frac{2}{\a} + n \left( \frac{1}{\a} - 1 \right) \left(\log n - 1 \right) \\
& - \frac{1}{2} \log n - \sum_{j=1}^{n} \log \overline{\s}(j).
\end{align*}
It is easy to see that
\begin{align}
\label{mod16}
F_{2}(n,\e, \overline{\s}) & \leq n \log \e + \int_{1}^{n} \log \k(x) dx + \log \k(n) \nonumber \\
& - 2 \left(\frac{1}{\a} - 1 \right) \log \sqrt{n} - \frac{1}{2} \log n - \frac{1-\a}{\a} \nonumber \\
& = n \log \left(\e \k(n) \right) + \log \k(n) - L(n) -
\frac{2-\a}{2\a} \ln n - \frac{1-\a}{\a} - \k(1).
\end{align}

From \eqref{mod15} and \eqref{mod16} we easily get \eqref{mod12}.
If $ \log \overline{\s}(x) $ is convex we use the following
estimate:

\begin{align*}
F_{2}(n,\e,\s) & \leq n \log \e - \int_{1}^{n} \log \overline{\s}(x) dx - \frac{1}{2} \log \overline{\s}(1) \\
& -\frac{1}{2} \log \overline{\s}(n) + 2 \left(\frac{1}{\a} - 1 \right) \int_{1}^{n} \log \sqrt{x} dx \\
& - \frac{1-\a}{\a} - \frac{1}{2} \log n + \int_{1}^{n} \log \left( \frac{2}{\a} \right)^{1/\a} dx + \log \left(\frac{2}{\a} \right)^{1/2} \\
& = n \log \e + \int_{1}^{n} \log \k(x) dx + \frac{1}{2} \log \frac{\k(n)}{n^{1/\a}} - \frac{1}{2} \log \overline{\s}(1) - \frac{1-\a}{\a} \\
& = n \log \e \k(n) - L(n) + \frac{1}{2\a} \log \frac{\left(\k(n)\right)^{\a}}{n} - \frac{1}{2} \log \overline{\s}(1) - \frac{1-\a}{\a}
 \end{align*}

and now  \eqref{mod13} easily follows. The proposition is proved.
\end{proof}

\begin{remark}
In the case $ \a = 2 $ we have  $ \k(x)=\left( \sqrt{x}
\overline{\s}(x) \right)^{-1} $, and again we can see that
estimates \eqref{mod12} and \eqref{mod13} coincide with the
corresponding estimates in \cite{hofmjorg}
\end{remark}

\begin{remark}
Estimates \eqref{mod3}, \eqref{mod5}, \eqref{mod12} and
\eqref{mod13} remain true if on the left hand sides of these
inequalities we put $ \sup_{a \in \H} \mu \left( S_{\e}(a) \right)
$ and $ \sup_{a \in \mu} \nu \left(S_{\e}(a) \right), $
respectively. In the Gaussian case in \cite{hofmjorg}, due to the
explicit expression of a Gaussian density it was possible to get
dependence on a in bounds for $ \mu\left( S_{\e}(a)\right) $.
\end{remark}

Now we are going to investigate the lower bounds for probabilities of stable measures on small sets. But at once we must say, that if
 in the case of upper bounds we were able to reach the same accuracy as in the Gaussian case, in the case of lower bounds the picture
 is quite different and we are facing principal difficulties. In order to obtain lower bounds we need independence of the coordinates
 of the H.r.v. under consideration. In the case of the Gaussian law and balls there is no restriction of independence since by means of
 an orthogonal transformation we can always get independence, but it is not so in the case of a stable law. The second difficulty arises
  when we use a moment inequality  of Cebyshev type. Thus, the absence of an explicit expression of a density of a stable law is more
   embarrassing when we consider lower bounds then the upper ones.

Therefore we are able to deal only with a stable measure $ \mu $
with ch.f. \eqref{mod1} and from now on we shall require the
additional assumption that $ \l \in l_{\b/\a}^{+} $ for some $ \b
< \a $. Let $ \overline{\s}_{\a}(x) $ be a differentiable
decreasing function such that $ \overline{\s}_{\a}(x) \geq
\l_{n}^{1/\a} $, $ n \geq 1 $, $ \rho(x) = C_{3}(\a) \left(
\overline{\s}_{\a}(x) \sqrt{x} \right)^{-1} $, where $ C_{3}(\a) =
2^{1/2} \G\left( \frac{1}{\a} \right) \sqrt{2} \left(\pi \a
\right)^{-1} $. Denote $$ M(y) = \int_{1}^{y} \frac{x
\rho'(x)}{\rho(x)} dx. $$ Let a quantity $ \overline{\eta}(\e, \b)
$ be defined so that

\begin{enumerate}
 \item[(i)] $ \e \rho \left(\overline{\eta} \left(\e,\b \right) \right) \geq \sqrt{2} $
 \item[(ii)] $ \sum_{i=n+1}^{\infty} \l_{i}^{\b/\a} < r \left( E \abs{\eta}^\b \right)^{-1} \e^{\b} $
\end{enumerate}
for all $ n > \overline{\eta}(\e, \b) $ and some $ 0 < r <
\frac{1}{2} \b/\a $, $ \b < \a $ (here $ \eta$ is real stable r.v.
with the ch.f. $ \exp \brc{-\frac{1}{2} \abs{t}^{\a}} $ ).

Further we define $ \rho_{1}(x) = C_{4}(\a) \wh{\s}_{\a}^{-1}(x) $
where $ C_{4}(\a) $ is such that $ p_{\a}(x) \geq C_{4}(\a) \min
\left(1, \abs{x}^{-(1+\a)}\right) $. Let $ M_{1}(y) = \int_{1}^{y}
x \rho'_{1}(x) \rho(x)^{-1} dx $ and let $
\overline{\eta}_{1}(\e,\b) $ be defined in the following manner:

\begin{enumerate}
 \item[(i)] $ \e \rho_{1} \left( \overline{\eta}_{1}(\e,\b) \right) > 1,$
 \item[(ii)] $ \sum_{i=n+1}^{\infty} \l_{i}^{\b/\a} < r \left( E\abs{\eta}^{\b}\right)^{-1} \e^{\b} $ for all $ n > \overline{\eta}_1(\e,\b) $
  and some $ 0 < r < 1 $ and $ \b < \a $,
 \item[(iii)] $\e^{-1} \wh{\s}_{\a} \left( \overline{\eta}_{1}(\e, \b) \right) \geq 1. $
\end{enumerate}
Note that the function $ \overline{\eta}_{1}(\e, \b) $ cannot be
defined for all $ \wh{\s}_{\a} $.

Let us put $ a(\a, n, \e) = \log \brc{C_{4}(\a) \min \left(1, \left(\dfrac{\sqrt{2} n^{1/\a}}{\e}\right)^{1+\a} \right) } $.
Now we are able to formulate lower bounds.

\begin{prop}
\label{prop3}
 For $ 0 < \a < 2 $ and all $ 0 < \e < 1 $ we have
 \begin{align}
 \label{mod17}
 \mu(S_{\e}) \geq C(r, \a, \l_{1}) \exp \brc{-M\left(\overline{\eta}\left(\e,\b\right)\right) - \frac{1}{2} \log \overline{\eta}(\e,\b) + a\left(\a, \overline{\eta}(\e, \b), \e\right) }
 \end{align}
and, if $ \overline{\eta}_{1}(\e, \b) $ exists, then
 \begin{align}
 \label{mod18}
 \mu(D_{\e}) \geq C(r, \a, \l_{i}) \exp \brc{-M \left( \overline{\eta}_{1}(\e,\b)
 \right)}.
 \end{align}
\end{prop}

\begin{proof}
Using the independence of the coordinates we get for all $ n \geq 1 $

\begin{align}
\label{mod19} \mu(S_{\e}) & = P \brc{\sum_{i=1}^{\infty}
\eta_{i}^{2} < \e^{2}} \geq P \brc{\sum_{i=1}^{n} \eta_{i}^{2} <
\e^{2}/2}P \brc{\sum_{i=n+1}^{\infty} \eta_{i}^{2} < \e^{2}/2} =
\nonumber \\
& \int_{\sum_{1}^{n} x_{i}^{2} \leq \frac{1}{2} \e^{2}}
\prod_{i=1}^{n} p_{\a,i}(x_{i}) \l_{n}(dx)
\left(1-P\left(\sum_{i=n+1}^{\infty} \eta_{i}^{2} \geq \frac{1}{2}
\e^{2} \right) \right)
\end{align}
and, taking some $\beta<\alpha,$ we have
\begin{align}
\label{mod20}
  P \brc{\sum_{i=n+1}^{\infty} \eta_{i}^{2} >  \e^{2}/2} \leq \left( \frac{2}{\e^{2}} \right)^{\b/2}
  E \absss{\sum_{i=n+1}^{\infty}\eta_{i}^{2}}^{\b/2}
  \leq \dfrac{2^{\b/2} E \abs{\eta}^{\b}} {\e^{\b}} \sum_{i=n+1}^{\infty}
  \l_{i}^{\b/2}.
\end{align}
It remains to estimate the integral in \eqref{mod19}:
\begin{align}
  & \int_{\sum_{1}^{n} x_{i}^{2} < \frac{1}{2} \e^{2}} \prod_{1}^{n} p_{\a,i}(x_{i}) \l_{n}(dx)
  \geq \left( \frac{\e}{\sqrt{2}} \right)^{n} V_{n} \min_{\sum_{1} x_{i}^{2} \leq \frac{1}{2} \e^{2}} \prod_{1}^{n} p_{\a,i}(x_{i}) \nonumber \\
 & = \frac{\e^{n}}{2^{n/2}} V_{n} \left( \prod_{1}^{n} \l_{i}^{1/\a} \right)^{-1} \min_{\sum_{1}^{n} x_{i}^{2} = \frac{1}{2}\e^{2}}
 \prod_{i=1}^{n} p_{\a} \left( \frac{x_{i}}{\l_{i}^{1/\a}}\right) \nonumber \\
 & = \frac{\e^{n}}{2^{n/2}} \frac{V_{n}}{\prod_{1}^{n} \l_{i}^{1/\a}} \min_{\sum_{1}^{n} y_{i}^{2} \l_{i}^{2/\a}=1}
  \prod_{1}^{n} p_{\a}(y_{i}) \nonumber \\
 \label{mod21}
 & = \frac{\e^{n}}{2^{n/2}} V_{n} \left( \prod_{i=1}^{n} \l_{i}^{1/\a}\right)^{-1} \left( p_{\a}^{(0)}\right)^{n-1}
  p_{\a} \left( \frac{\e}{\sqrt{2} \l_{n}^{1/\a}}\right)
\end{align}
From \eqref{mod19} - \eqref{mod21}, having  in mind (ii), we get
\begin{align*}
 \mu(S_{\e}) & \geq \left(1-2^{\b/2} r \right) 2^{-n/2} \e^{n} V_{n} \left( \prod_{1}^{n} \l_{i}^{1/\a} \right)^{-1} \\
 & \left(p_{\a}(0)\right)^{n-1} p_{\a} \left(\e \left(\sqrt{2} \l_{n}^{1/\a} \right)^{-1}
 \right).
\end{align*}
As above, using Stirling's formulae, after some calculations,  we
arrive at the estimate

\begin{align*}
\mu(S_{\e}) & \geq C(\a, r) \exp \Big \{n \log \frac{\e}{\sqrt{2}} - \sum_{i=1}^{n} \log \wh{\s}_{\a}(i) + n \log C_{3}(\s) \\
&  - \log \sqrt{n} - \frac{1}{2} n(\log n -1) + a(\a, n,\e)\Big \} \\
& \geq C(\a,r) \exp \Big \{n \log \frac{\e}{\sqrt{2}} - \int_{1}^{n} \log \wh{\s}_{\a}(x) dx -\log\wh{\s}_{\a}(1) \\
& - \log \sqrt{n} \int_{1}^{n} \log C_{3}(\a) + \log C(\a)
- \int_{1}^{n} \log \sqrt{x} dx + \frac{1}{2} + a(\a, n, \e)\Big \} \\
& \geq C(\a, r, \l_{1}) \exp \brc{n \log \frac{\e}{\sqrt{2}}
\rho(n) - M(n) - \frac{1}{2} \log n + a(\a, n, \e)},
\end{align*}
which gives us \eqref{mod17}.

The proof of \eqref{mod18} is simpler. We have

\begin{align*}
\mu(D_{\e}) & = \prod_{i=1}^{n} p\brc{\abs{\eta_{i}} > \e} p \brc{\abs{\eta_{k}}>\e, k \geq n+1} \\
& \e^{n} \prod_{i=1}^{n} p_{\a,i}(\e) \left(1- \sum_{i=n+1}^{\infty} p \brc{\abs{\eta_{i}} > \e} \right) \\
& = \e^{n} \left( \prod_{1}^{n} \l_{i}^{1/\a} \right)^{-1} \prod_{i=1}^{n} p_{\a} \left( \frac{\e}{\l_{i}^{1/\a}} \right)
 \left(1 - \e^{-\b} E \abs{\eta}^{\b} \sum_{i=n+1}^{\infty} \l_{i}^{\b/\a}
 \right).
\end{align*}
Since in the following we shall choose $ n $ equal to $ \overline{\eta}_{1}(\e, \b) $, we can use (iii) from which it follows,
that for all $ 1 \leq i \leq \overline{\eta}_{1}(\e, \b) $ we have $ p_{\a} \left( \e \l_{i}^{-1/\a} \right) > C_{4}(\a) $. Then we get

\begin{align*}
\mu(D_{\e}) \geq C(r) \exp \brc{n \log \e + n \log C_{4}(\a) - \sum_{i=1}^{n} \log \wh{\s}_{\a}(i)}
\end{align*}
It is obvious, that from this estimate we can get \eqref{mod18}.
The proposition is proved.
\end{proof}

At the end of this section we give some examples of upper and
lower bounds in the case where we have a given sequence $ \l $. We
omit all calculations and give only final results. As above, $ \mu
\equiv \mu(\a, \l) $ - stable distribution with the ch.f.
\eqref{mod1} and $ \nu = \nu(\a, \l) $ - stable distribution with
the ch. f. \eqref{mod11}.

\begin{example}
Let $ \l_{i} = i^{-\g} $, $ i \geq 1 $, $ \g > 1 $. Then from proposition \ref{prop1}, \ref{prop2}, \ref{prop3} we get
\begin{align}
\label{mod22} \nu\left(S_{\e}(a)\right) \leq C(\a,\g) \e^{-\frac{\a
(\g-2)}{2(2-2\a+\g\a)}} \exp \brc{-C_{5}(\a,\g)
\e^{-\frac{2\a}{2-2\a+\g\a}}},
\end{align}

\begin{align*}
C_{5}(\a,\g) = \dfrac{2-2\a+\g\a }{2\a} \left(\frac{\a}{2}
\right)^{2/(2-2\a+\g\a)},
\end{align*}

\begin{align}
\label{mod23} \mu(D_{\e}) \leq C \frac{1}{\sqrt{\e}} \exp \brc{-
\frac{\pi \g}{2^{1/\a} \G\left(\frac{1}{\a}\right)}
\e^{-\frac{\a}{\g}}},
\end{align}

\begin{align}
\label{mod24} \mu(S_{\e}(a)) \leq C(\a, \g)
\e^{\frac{\a-\g}{2\g-\a}} \exp \brc{-C_{6}(\a,\g)
\e^{-\frac{2\a}{2\g-\a}}},
\end{align}

\begin{align*}
C_{6}(\a,\g) = \frac{2\g-\a}{2\a} \left( \dfrac{\sqrt{\pi} \a}{2
\G\left(\frac{1}{\a}\right)} \right)^{2\a/2\g-\a},
\end{align*}

\begin{align}
\label{mod25} \mu(S_{\e}) \geq C(\a, \d, \g)
\e^{\frac{1+\a+\d/2}{\g^{\d-1}}} \exp \brc{-C_{7}(\a, \g)
\e^{-\frac{\d\a}{\g\d-1}}},
\end{align}
where $ \d=\g_{1}/\g $, $ \g_{1} $ - any number, satisfying $ 1 <
\g_{1} < \g, $ and $ C_{7}(\a, \g) = (2\g-\a)(2\a)^{-1} $.

One can easily verify, that in the case $ \a=2 $ the estimates
\eqref{mod22} and \eqref{mod24} coincide with that one, given in
example 4.1 of \cite{hofmjorg}. We cannot apply \eqref{mod18}
since in this case $ \overline{\eta}_{1}(\e, \b) $ is not defined.
Namely, if $ \wh{\s}(x) = x^{-\g/\\a} $, then in order to satisfy
(ii) $ \overline{\eta}_{1}(\e, \b) $ must be $ \geq C
\e^{-\frac{\a\b}{\g\b-1}} $, where $ \b $ is such that $ \g\b > \a
$, $ \b < \a $, but then (iii) is violated.
\end{example}

\begin{example}
Let $ \l_{i} = e^{-i} $, $ i \geq 1 $. In this example we find $ \overline{\eta}_{1}(\e,\b) $ and apply \eqref{mod18}. We get
\begin{align}
\label{mod26} \mu(D_{\e}) \geq C(\a, \b, r) \exp \brc{-\a \ln^{2}
\frac{1}{\e}}.
\end{align}
When we apply the estimate \eqref{mod5} we find

\begin{align}
\label{mod27} \mu(D_{\e}) \leq C(\a) \e^{-C_{8}(\a)} \exp
\brc{-\frac{\a}{2} \ln^{2} \frac{1}{\e}},
\end{align}
\begin{align*}
C_{8}(\a) = \frac{1}{2} + 2 \a^{2} \ln
\left(\pi\a\left(2^{\frac{1+\a}{\a}} \G\left(\frac{1}{\a} \right)
\right)^{-1} \right).
\end{align*}
The comparison of the estimates \eqref{mod24} - \eqref{mod27}
shows us that there is rather great difference between upper and
lower estimation, and this is the main reason, why we cannot
construct estimates of the quantity $ \abs{F-\mu}(H) $ by means of
pseudo moments, as it was done in the Gaussian case in
\cite{paulII}.
\end{example}

\vskip 2 cm

\footnotesize
\begin{minipage}[t]{90mm}{\scshape
  Department of Mathematics and Informatics,\\
  Vilnius University\\
  Naugarduko 24, 03225 Vilnius, Lithuania\\
  E-mail:} \texttt{vygantas.paulauskas@mif.vu.lt}
\end{minipage}

\end{document}